\documentclass{amsart}
\usepackage{etoolbox,expl3,xparse}
\usepackage[T1]{fontenc}
\usepackage[utf8]{inputenc}
\usepackage{xcolor}
\usepackage{caption}
\usepackage[shortlabels,inline]{enumitem}
\usepackage[leqno]{mathtools}
\usepackage{amssymb}
\usepackage{mleftright}
\usepackage{derivative}
\usepackage{leftindex}
\usepackage{tikz}
\usepackage{hyperref}
\usepackage[capitalize,nameinlink]{cleveref}

\hypersetup{%
  bookmarksnumbered=true,%
  colorlinks=true,%
  linkcolor=blue,%
  citecolor=blue,%
  filecolor=blue,%
  menucolor=blue,%
  urlcolor=blue,%
  pdfnewwindow=true,%
  pdfstartview=FitBH}

\numberwithin{equation}{section}

\newtheorem{thm}{Theorem}[section]
\newtheorem{prop}[thm]{Proposition}
\newtheorem{cor}[thm]{Corollary}
\newtheorem{lem}[thm]{Lemma}

\theoremstyle{definition}
\newtheorem{eg}[thm]{Example}
\newtheorem{defn}[thm]{Definition}
\newtheorem*{remark}{Remark}

\usetikzlibrary{cd} 
\usetikzlibrary{arrows}
\usetikzlibrary{backgrounds}
\usetikzlibrary{calc}
\usetikzlibrary{decorations}
\usetikzlibrary{shapes}

\DeclarePairedDelimiterX\abs[1]\lvert\rvert
  { \ifblank{#1}{\:\cdot\:}{#1} }
\DeclarePairedDelimiterX\angles[1]\langle\rangle
  { \ifblank{#1}{\:\cdot\:}{#1} }
\DeclarePairedDelimiterX\norm[1]\lVert\rVert
  { \ifblank{#1}{\:\cdot\:}{#1} }
\DeclarePairedDelimiterX\variables[1]\lparen\rparen
  { \ifblank{#1}{\:\cdot\:}{#1} }
\DeclarePairedDelimiterX\ceil[1]\lceil\rceil
  { \ifblank{#1}{\:\cdot\:}{#1} }
\DeclarePairedDelimiterX\floor[1]\lfloor\rfloor
  { \ifblank{#1}{\:\cdot\:}{#1} }
\DeclarePairedDelimiterX\bra[1]\langle\vert
  { \ifblank{#1}{\:\cdot\:}{#1} }
\DeclarePairedDelimiterX\ket[1]\vert\rangle
  { \ifblank{#1}{\:\cdot\:}{#1} }
\DeclarePairedDelimiterX\braket[2]\langle\rangle
  {
    \ifblank{#1}{\:\cdot\:}{#1}
    \,\delimsize\vert\,\mathopen{}
    \ifblank{#2}{\:\cdot\:}{#2}
  }
\DeclarePairedDelimiterX\pairing[2]\langle\rangle
  { \ifblank{#1}{\:\cdot\:}{#1}, \ifblank{#2}{\:\cdot\:}{#2} }
\DeclarePairedDelimiterX\inner[2]\lparen\rparen
  { \ifblank{#1}{\:\cdot\:}{#1}, \ifblank{#2}{\:\cdot\:}{#2} }

\DeclarePairedDelimiterX\Set[1]\{\}
  { #1 }
\DeclarePairedDelimiterX\GSet[1]\langle\rangle
  { #1 }

\NewDocumentCommand \fun { m e{^_} O{} }
  {%
    \operatorname{#1}%
    \IfValueT{#2}{\sp{#2}}%
    \IfValueT{#3}{\sb{#3}}%
    \ifblank{#4}{}{\mleft(#4\mright)}%
  }

\NewDocumentCommand \mcal { m } {\fun{\mathcal{#1}}}
\NewDocumentCommand \mscr { m } {\fun{\mathscr{#1}}}

\NewDocumentCommand \mbf { m } {\fun{\mathbf{#1}}}

\NewDocumentCommand \cat { m } {\mbf{#1}}

\ExplSyntaxOn
\NewDocumentCommand{\createbunch}{ m O{} m }
 {
  \clist_map_inline:nn { #3 } { \cs_new_protected:cpn { #2 ##1 } { #1 { ##1 } } }
 }
\ExplSyntaxOff

\createbunch{\mathbb}{N,Z,Q,R,C,F,K,T,D,W,M}
\createbunch{\mcal}[c]{A,C,E}
\createbunch{\mscr}[s]{A,C,E}
\NewDocumentCommand \one {} {\mathbf{1}}

\NewDocumentCommand \Hopf {} {{H}}
\NewDocumentCommand \dHopf {} {\dual{\Hopf}}

\NewDocumentCommand \stab {m} {\cat{stab}[#1]}
\NewDocumentCommand \modcat { m } 
    {\cat{mod}[#1]}
\NewDocumentCommand \comod { m } 
    {\cat{comod}[#1]}
\NewDocumentCommand \modp { m } {\cat{mod}^{\rm p}[#1]}
\NewDocumentCommand \comodp { m } {\cat{comod}^{\rm p}[#1]}

\createbunch{\ensuremath\fun}{Hom,End,Aut,Ext,Tor,Cotor,Ker,Coker,coker,Id,id,tr,pt,Ob,Mor,Pic,Gal,Mod,Gr,gr,Ind,Proj,Ho,Sq}

\NewDocumentCommand \iso {} {\cong}
\NewDocumentCommand \tensor{ om }
    {\leftindex_{#1}\otimes_{#2}}
\NewDocumentCommand \dual { m }
    { \ifblank{#1}{(\:\cdot\:)}{#1}^{\vee} }

\NewDocumentCommand \opp { m }
    { \ifblank{#1}{(\:\cdot\:)}{#1}^{\rm op} }

\DeclareRobustCommand\longepimorphism
	{\relbar\joinrel\twoheadrightarrow} 
\NewDocumentCommand\longepi{}{\longepimorphism}
\DeclareRobustCommand\longmonomorphism
	{\lhook\joinrel\longrightarrow} 
\NewDocumentCommand\longmono{}{\longmonomorphism}

\NewDocumentCommand \mm {}{\mathfrak{m}}

\NewDocumentCommand \Rmod {} {\modcat{R}}
\NewDocumentCommand \Hmod {} {\modcat{\dHopf}}
\NewDocumentCommand \Amod {} {\modcat{A}}
\NewDocumentCommand \Hcomod {} {\comod{\Hopf}}
\NewDocumentCommand \Hmodp {} {\modp{\dHopf}}
\NewDocumentCommand \Hcomodp {} {\comodp{\Hopf}}
\NewDocumentCommand \cAC {O{1}} {\cA(#1)^{\C}}
\NewDocumentCommand \cAR {O{1}} {\cA(#1)^{\R}}
\NewDocumentCommand \dcAR {O{1}} {\dual{\cA(1)}_{\R}}
\NewDocumentCommand \MtwoC {O{2}} {\M_{#1}^{\C}}
\NewDocumentCommand \MtwoR {O{2}} {\M_{#1}^{\R}}

\usepackage[initials]{amsrefs}
\BibSpec{arXiv}{%
  +{}{\PrintAuthors}{author}
  +{,}{ \textit}{title}
  +{,}{ \arXiveprint}{eprint}
  +{}{ \parenthesize}{date}
}
\NewDocumentCommand \arXiveprint {m} {\href{https://arxiv.org/abs/#1}{arXiv:#1}}

\title[The stable Picard group of $\R$-motivic finite free Hopf algebra]{The stable Picard group of finite Adams Hopf algebroids with an application to the $\R$-motivic Steenrod subalgebra $\cAR$}

\author[Gao]{Xu Gao}
\address[Gao]{Department of Mathematics, University of California, Santa Cruz}
\author[Li]{Ang Li}
\address[Li]{Department of Mathematics, University of California, Santa Cruz}

\begin{document}
\begin{abstract}
	In this paper, we investigate the rigidity of the stable comodule category of a specific class of Hopf algebroids known as \emph{finite Adams}, shedding light on its Picard group. 
	Then we establish a reduction process through base changes, enabling us to effectively compute the Picard group of the  \emph{$\R$-motivic mod $2$ Steenrod subalgebra} $\cAR$. 
        Our computation shows that $\Pic(\cAR)$ is isomorphic to $\Z^4$, where two ranks come from the motivic grading, one from the algebraic loop functor, and the last is generated by the \emph{$\R$-motivic joker} $J$.
\end{abstract}
\maketitle

\tableofcontents


\section{Introduction}\label{sec:intro}

\subsection{The Picard group of a Hopf algebra}
Given a monoidal category $\mathcal{C}$, its \emph{Picard group} is defined as the set of invertible isomorphism classes in $\mathcal{C}$, equipped with the multiplication induced by the tensor product. 

A notable example is the \emph{(stable) Picard group} $\Pic(A)$ of a \emph{finite-dimensional cocommutative Hopf algebra} $A$ over a field $k$. 
That is defined as the Picard group of the \emph{stable module category} $\stab{A}$.
In $\stab{A}$, objects are finitely generated left $A$-modules, and morphisms are $A$-module homomorphisms modulo those factoring through projective $A$-modules.  
The symmetric monoidal structure is given by the \emph{smash product}. 
For an $A$-module $M$, its $k$-dual $\dual{M}:=\Hom_k(M,k)$ naturally becomes a comodule over the dual coalgebra $\dual{A}$ of $A$. 
Furthermore, the Hopf algebra structure of $A$ ensures that any $\dual{A}$-comodule is equipped with an $A$-module structure. 
Consequently, dualization provides us a contravariant auto-equivalence on $\Amod$, which passes to the stable module category
\[
	D\colon\opp{\stab{A}} \longrightarrow \stab{A}.
\]
Therefore, $\stab{A}$ is a \emph{rigid category} allowing us to determine whether an object in $\stab{A}$ is invertible by checking if it is \emph{endo-trivial}. 
Some elements of $\Pic(A)$ are readily identified: such as the possible grading shiftings of the tensor unit. 
Additionally, the category $\stab{A}$ is triangulated, and the algebraic loop functor also contributes to $\Pic(A)$. 
However, these elements may not account for the entire Picard group.
The cokernel of the inclusion is thus the interesting part of investigating.

\subsection{The Picard group of $\cA(1)$}
The \emph{mod $2$ Steenrod algebra} $\cA$ is a Hopf algebra over $\F_2$, generated by the \emph{Steenrod squares} $\Sq^{2^i}$ ($i\geqslant 0$), where $\Sq^n$ represents the cohomology operation of degree $n$. 
The whole algebra $\cA$ is not finite-dimensional, while the
subalgebra $\cA(n)$ generated by $\Sq^1, \Sq^2, \cdots, \Sq^{2^n}$ is finite. One can read more in \cite{margolis}*{Part II}.

The Picard group of $\cA(1)$ was computed by Adams and Priddy\cite{adams1976uniqueness}, where the interesting cokernel is an order $2$ torsion generated by an $\cA(1)$-module called the \emph{joker} (see \cref{fig:jokerc}). The Joker is the cyclic $\cA(1)$-module $\cA(1)/\Sq^3\cA(1)$ which plays a special role in the study of $\cA(1)$-modules. For a recent result which highlights the special significance of the Joker see \cite{bhattacharya2017stable}.
\begin{figure}[h]
	\[
		\begin{tikzpicture}
		\begin{scope}[ thick, every node/.style={sloped,allow upside down}, scale=0.5]
		\draw (3,0)  node[inner sep=0] (v30) {} -- (3,1) node[inner sep=0] (v31) {};
		\draw  (2.5,2) node[inner sep=0] (v22) {};
		\draw (3,3)  node[inner sep=0] (v33) {} -- (3,4) node[inner sep=0] (v34) {};
		\draw [color=blue] (v31) to [bend right=70] (v33);
		\draw [color=blue] (v30) to [bend left = 50] (v22);
		\draw [color=blue] (v22) to [bend left =50] (v34);
		\filldraw (v30) circle (2.5pt);
		\filldraw (v31) circle (2.5pt);
		\filldraw (v22) circle (2.5pt);
		\filldraw (v33) circle (2.5pt);
		\filldraw (v34) circle (2.5pt);
		\draw (3,-0.5) node[left]{} (3,1) node[right]{} (2.5,2) node[left]{} (1,2) node[left]{$ $};
		\draw (3,4) node[right]{$ $} (3,3) node[right]{$ $} (2,3) node[left]{$ $} (2,2) node[right]{$ $};
		\end{scope}\end{tikzpicture}
	\]
	\caption{The $\cA(1)$-module Joker is illustrated as above. Each $\bullet$ represents an $\F_2$-generator. The black and blue lines represent the action of $\Sq^1$ and $\Sq^2$, reading from bottom to top, respectively.}
	\label{fig:jokerc}
\end{figure}
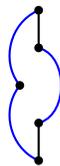

Voevodsky constructed the motivic analogue of the Steenrod (sub)algebra\cite{voevodsky2003reduced}. Gheoghe, Isaksen, and Ricka \cite{gheorghe2018picard} thus computed the Picard group of the complex motivic Steenrod subalgebra $\cAC$. Likewise, the interesting cokernel is still generated by the \emph{$\C$-motivic joker}, although it is no longer torsion.

\subsection{The $\R$-motivic situation}
In both classical and $\C$-motivic cases, the Steenrod subalgebra being considered is a Hopf algebra. One major difference is, the \emph{$\C$-motivic Steenrod subalgebra} $\cAC$ is free of finite rank over $\MtwoC$, the coefficient ring of the \emph{$\C$-motivic cohomology theory} (cf. \cref{sec:mot}), rather than over $\F_2$. Unlike its $\C$-motivic counterpart, the \emph{$\R$-motivic Steenrod subalgebra} $\cAR$ is not a Hopf algebra. 
However, it remains free of finite rank over the coefficient ring $\MtwoR$. Its $\MtwoR$-dual, turns out to be a \emph{Hopf algebroid} with base $\MtwoR$ in the sense of \cite{ravenel2023complex}*{A1} and \cite{hovey2003homotopy}.

Over a Hopf algebroid, the dualization functor is not an endo functor on $\stab{A}$. However, many results concerning finite-dimensional cocommutative Hopf algebras can be extended to the Hopf algebroid setting. In this paper, we investigate the notions of the \emph{stable (co)module category} and the \emph{Picard group} and constructed a functor behaves analogously as the dualization functor. Even better, since $\MtwoR$ is a polynomial ring, the stable module category of $\cAR$ comes from a \emph{Frobenius exact category} (cf. \cref{Frobenius}) and is thus \emph{rigid}. 
Then, following standard arguments on a rigid symmetric monoidal category, we are able to determine the part of the Picard group coming from grading shifting of the unit and the algebraic loop functor.

The computation of the rest part can be achieved through a reduction technique utilizing a base change functor from $\MtwoR$ to $\F_2$. This allows us to return to the more familiar realm of finite-dimensional cocommutative Hopf algebras over $\F_2$.
The key observation is that the base change functor between the relevant exact categories behaves very well, enabling us to pull back information regarding the Picard group.

For $\cAR$, its base change to $\F_2$ is isomorphic to the group algebra $\F_2[D_8]$ of the \emph{dihedral group} $D_8$ of order $8$, which is well-understood. 
Our computation shows that the Picard group of $\cAR$ is isomorphic to that of $\cAC$ and of $\F_2[D_8]$. 

However, the assumption that the base ring of the Hopf algebroid is a polynomial ring needs not to be satisfied. 
One notable example is the \emph{$C_2$-equivariant cohomology over a point}. 
Consequently, the results presented in this paper may not be directly applicable to its $C_2$-equivariant counterpart.

\subsection{Organization.}
In \cref{sec:background}, we delve into the stable module category of the dual algebra of a Hopf algebroid. 
We reveal it as the quotient of a Frobenius exact category and then deduce the ordinary part of the Picard group. 
Next, in \cref{sec:mot}, we introduce the $\R$-motivic setting and construct a base change functor to the classical setting. 
We establish a comparsion injection between the Picard groups and employ the Margolis homology as a tool for assessing invertibility. 
Subsequently, in \cref{sec:pic}, we address the results in previous sections and conduct a comprehensive computation of the Picard group of $\cAR$.
Finally, in \cref{sec:final}, we outline the potential generalizations and applications stemming from the present results.


\section{Stable module theory of finite Adams Hopf algebroids}\label{sec:background}
Let's set up some categories we will be working with:
\begin{itemize}[wide]
	\item $(\K,\tensor{},\one)$: a \emph{locally noetherian symmetric monoidal abelian category}. 
	That is a symmetric monoidal abelian category satisfying the \emph{AB5} axiom, the tensor functor $\tensor{}$ commutes with colimits, and $\K$ has a generating set of \emph{noetherian objects}, namely objects satisfying the \emph{ascending chain condition}.
	
\begin{eg}
	The category of modules over a noetherian ring $k$ provides a locally noetherian symmetric monoidal abelian category.
	Its noetherian objects are the finitely generated $k$-modules.
\end{eg}
\begin{remark}
    We will use the notion of \emph{generalized elements} of objects $X$ of $\K$. 
    That are sections of the presheaf $\Hom(-,X)$.
    Then the \emph{Yoneda Lemma} tells us that an object is determined by its generalized elements.
\end{remark}
	\item $\cat{Alg}_{\K}$: the category of monoids in $\K$. 
	Its objects are called \emph{$\K$-algebras}. 
	Its morphisms are called \emph{$\K$-algebraic maps}. 
	\item $\cat{Com}_{\K}$: the full subcategory of $\cat{Alg}_{\K}$ consisting of \emph{commutative} $\K$-algebras.

	\item $\cat{Mod}[R]$: the symmetric monoidal abelian category of 
	left $R$-modules for a commutative $\K$-algebra $R$. 
	Its tensor product structure is denoted by $-\tensor{R}-$. 
	Note that $\cat{Mod}[R]$ is a closed category due to the commutativity of $R$.
	We will use $\dual{M}$ to denote the \emph{weak dual} $\Hom_R(M,R)$ of an object $M$ in $\cat{Mod}[R]$. 
	The natural pairing of $\dual{M}$ and $M$ is denoted by $\braket{-}{-}$. 
\begin{remark}
	We still need to consider two different $R$-module structures on the same underlying object $M$. 
	To distinguish them, we refer to them as the \emph{left} $R$-module structure (abbreviated as $l$) and the \emph{right} $R$-module structure (abbreviated as $r$), even though both $(M,l)$ and $(M,r)$ are treated as objects in $\cat{Mod}[R]$.
	We will use $M\tensor{r,l}N$ to denote the tensor product of the $R$-modules $(M,r)$ and $(N,l)$. 
	The additional $R$-module structure $r$ on $M$ gives the weak dual $\dual{M}$ of $(M,l)$ an extra \emph{right} $R$-module structure\footnote{%
		Which is exactly the \emph{left} $R$-module structure on $\dual{M}$ when $M$ is viewed as an $(R,R)$-bimodule.} 
	given by the formula 
	$\braket{-\cdot r}{-}\coloneq\braket{-}{-\cdot r}$.
\end{remark}
	\item $\Rmod$: the full subcategory of $\cat{Mod}[R]$ consisting of noetherian $R$-modules.	  
\end{itemize}

The \emph{Hilbert's Basis Theorem} \cite{HBTrev} tells us that any finitely generated $\K$-algebra $R$ is \emph{noetherian} in the sense that $\cat{Mod}[R]$ is a locally noetherian symmetric monoidal abelian category. 
If this is the case, then $\Rmod$ consists of \emph{finitely generated} $R$-modules (namely quotients of the free $R$-modules over noetherian objects in $\K$). In what follows, we keep this assumption.

\subsection{Hopf algebroid and its comodules}
Following \cite{ravenel2023complex}*{A1} and \cite{hovey2003homotopy}, we give the following definition.
\begin{defn}
	A \emph{Hopf algebroid} is a cogroupoid object in the category $\cat{Com}_{\K}$. 
\end{defn}
Spelling out the above definition, a \emph{Hopf algebroid} consists of a pair $(\Hopf,R)$ of commutative $\K$-algebras and the following $\K$-algebraic maps
\begin{align*}
	\Hopf &\overset{\Delta}{\longrightarrow}
		\Hopf\tensor{r,l}\Hopf 
		&&\text{\emph{comultiplication} inducing the composition}\\
	\Hopf &\overset{c}{\longrightarrow} 
		\Hopf 
		&&\text{\emph{conjugation} inducing the inverse}\\
	\Hopf &\overset{\epsilon}{\longrightarrow} 
		R 
		&&\text{\emph{augmentation} inducing the identity}\\
	R &\overset{\eta}{\longrightarrow} 
		\Hopf 
		&&\text{\emph{right unit} inducing the target}\\
	R &\overset{\iota}{\longrightarrow} 
		\Hopf 
		&&\text{\emph{left unit} inducing the source}
\end{align*}
satisfying certain axioms. 
\begin{remark}
	The $\K$-algebra $R$ is called the \emph{base} of this Hopf algebroid. 
	When the base is implicated, we may simply say that $\Hopf$ is a Hopf algebroid. 
	The $\K$-algebraic maps $\iota$ and $\eta$ provide the $\K$-algebra $\Hopf$ two distinct $R$-module structures. 
	We will refer to the one induced by $\iota$ as the \emph{left} $R$-module structure, and the one induced by $\eta$ as the \emph{right} $R$-module structure. 
	Then the $\K$-algebraic maps $\Delta$ and $\epsilon$ are $R$-linear for both $R$-module structures, and $c$ is an anti-isomorphism between the \emph{left} and \emph{right} $R$-module structures. 
	We will employ the \emph{Sweedler's notation}, denoting the comultiplication simply as $\Delta(x)=x_{(1)}\tensor{} x_{(2)}$. 
\end{remark}

\begin{defn}
	A \emph{(left) $\Hopf$-comodule} is a left $R$-module $M$ with an $R$-linear map 
	\[
			\psi_M\colon M\longrightarrow \Hopf\tensor{R}M
	\]
	satisfying the axioms of counts and coassociativity.
\end{defn}
\begin{remark}
	We will employ the \emph{Sweedler's notation}, denoting the coaction simply as $\psi_M(m)=m_{[-1]}\tensor{} m_{[0]}$. 
	Please note that: 
	\begin{enumerate}
		\item $m_{[-1]}$ presents the $\Hopf$-parts of $\psi_M(m)$; and 
		\item $m_{[0]}$ presents the $M$-parts of $\psi_M(m)$.
	\end{enumerate}
\end{remark}

\begin{defn}
	Given two $\Hopf$-comodules $M,N$, their \emph{tensor product} is the left $R$-module $M\tensor{R}N$ with the coaction
	\[
			\psi_{M\tensor{R}N}(m\tensor{} n) = 
			m_{[-1]}n_{[-1]} \tensor{} m_{[0]} \tensor{} n_{[0]}.
	\]
\end{defn}

We will use $\Hcomod$ to denote the category of finitely generated $\Hopf$-comodules and $\Hcomodp$ its full subcategory consisting of $\Hopf$-comodules whose underlying $R$-modules are finitely generated and projective.
\begin{prop}
	The category $\Hcomod$ is a symmetric monoidal category.
\end{prop}
\begin{proof}
	The tensor product is given as above. The unit of this monoidal structure is $R$, with the $\Hopf$-comodule structure given by $\iota$.
\end{proof}
\begin{prop}
	If $\Hopf$ is flat over $R$, then the $R$-linear category $\Hcomod$ is an abelian category and $\Hcomodp$ is an exact full subcategory of it.
\end{prop}
\begin{proof}
	By \cite{ravenel2023complex}*{Theorem A1.1.3}, $\Hcomod$ is an abelian category and the forgetful functor from $\Hcomod$ to $\modcat{R}$ is exact. 
	Then the second statement follows from standard results on $\modcat{R}$.
\end{proof}

\begin{defn}
	The \emph{stable comodule category} of $\Hopf$ is the category $\stab{\Hopf}$ whose objects are the same as in $\Hcomodp$ and whose morphisms are given by
	\[
	\Hom_{\stab{\Hopf}}[M,N]\coloneq\Hom_{\Hopf}[M,N]/\sim,
	\]
	where two morphisms are equivalent if their difference factors through a projective $\Hopf$-comodule. 
\end{defn}

\begin{remark}
	Let's justify why we consider $\Hcomodp$ instead of $\Hcomod$. In fact, the monoidal category $\Hcomod$ is closed \cite{hovey2003homotopy}*{Theorem 1.3.1}.
		Thus, we are able to talk about \emph{dualizable} objects. It turns out that an $\Hopf$-comodule is dualizable if and only if it is finitely generated and projective as an $R$-module \cite{hovey2003homotopy}*{Proposition 1.3.4}.
	However, in this paper, we would rather to take a more roundabout, yet explicit, approach to involve modules over the dual algebra $\dHopf$, seeing \ref{dualization}.
\end{remark}

\subsection{The dual of a Hopf algebroid}
Following \cite{hovey2003homotopy}*{Definition 1.4.3}, we say a Hopf algebroid $(\Hopf,R)$ is \emph{finite Adams} 
if $\Hopf$ is finitely generated and projective as a left $R$-module. We will keep this assumption in what follows. Then the $R$-coalgebra structure $(\Delta,\epsilon)$ on $\Hopf$ induces an $R$-algebra structure $(\mu,\dual{\epsilon})$ on $\dHopf$, where the multiplication $\mu\colon\dHopf\tensor{R}\dHopf\to\dHopf$ is given by the formula
\[	
	f\tensor{} g
	\longmapsto
	\braket*{f}{-_{(1)}\cdot\eta\braket{g}{-_{(2)}}}.
\]

We will use $\Hmod$ to denote the abelian category of finitely generated (left) $\dHopf$-modules and $\Hmodp$ its full exact subcategory consisting of $\dHopf$-modules whose underlying $R$-modules are finitely generated and projective.
\begin{eg}
	A commutative Hopf algebra $\Hopf$ over a $\K$-algebra $R$ gives a Hopf algebroid $(\Hopf,R)$ with identical left and right units. 
	Note that if $\Hopf$ is finitely generated and projective as a left $R$-module, then $\dual{\Hopf}$ is a \emph{cocommutative} Hopf algebra over $R$.
\end{eg}
\begin{eg}
	In our application, the Hopf algebroid $(\Hopf,\MtwoR)$ will be given by the $\MtwoR$-dual of the \emph{$\R$-motivic Steenrod subalgebra} $\cAR[n]$.
\end{eg}
\begin{defn}
	The \emph{stable module category} of $\dHopf$ is the category $\stab{\dHopf}$ whose objects are the same as in $\Hmodp$ and whose morphisms are given by
	\[
	\Hom_{\stab{\dHopf}}[M,N]\coloneq\Hom_{\dHopf}[M,N]/\sim,
	\]
	where two morphisms are equivalent if their difference factors through a projective $\dHopf$-module. 
\end{defn}

\subsection{Comodule-module equivalence}\label{comodule=module}
The categories we have introduced so far are related by the following proposition:
\begin{prop}
	The two $R$-linear abelian categories $\Hcomod$ and $\Hmod$ over the base category $\Rmod$ are isomorphic.
	\[
	\begin{tikzcd}
		\Hcomod 
		\arrow[dr,"\text{forgetful}"']\arrow[rr, dotted]&& 
		\Hmod \arrow[dl,"\text{forgetful}"] \\
		&\Rmod&
	\end{tikzcd}
	\]
\end{prop}
\begin{proof}
	First note that 
	\begin{enumerate}[wide]
		\item There is an $R$-bilinear \emph{evaluation map} $\dHopf\tensor{R}\Hopf\to R$ given by $\braket*{-}{c(-)}$. 
		\item Since $\Hopf$ is finitely generated and projective as an $R$-module, we also have a \emph{unit} $R\to\Hopf\tensor{R}\dHopf$. 
		We will employ the \emph{Sweedler's notation}, denoting the image of $1\in R$ simply as $1_{(-)}\tensor{} 1_{(+)}$.
	\end{enumerate}
	
	Then the isomorphism is given as follows.
	\begin{enumerate}[wide]
		\item Any $\Hopf$-comodule $M$ is equipped with an $\dHopf$-module structure 
		\[
			\dHopf\tensor{R}M\longrightarrow M
			\colon
			f\tensor{} m\longmapsto
			\braket*{f}{c(m_{[-1]})}\tensor{} m_{[0]}.
		\]
		We will denote this $\dHopf$-module by $c(M)$.
		\item Conversely, any $\dHopf$-module $M$ admits an $\Hopf$-comodule structure 
		\[
			M\longrightarrow \Hopf\tensor{R}M
			\colon
			m\longmapsto 1_{(-)}\tensor{} 1_{(+)}.m
		\]
		We will denote this $\Hopf$-comodule by $c(M)$. 
		\item Then it is straightforward to verify that $c\circ c=\id$.\qedhere
	\end{enumerate}
\end{proof}

We are thus able to further define a symmetric monoidal structure on $\Hmod$ by transplanting the one on $\Hcomod$.

\begin{remark} 
The provided isomorphism $c$ also identifies the full exact subcategories $\Hcomodp$ and $\Hmodp$. 
Consequently, it induces an equivalence between $\stab{\Hopf}$ and $\stab{\dHopf}$. 
We will mainly focus on the latter.
\end{remark}
\begin{prop}
	The category $\stab{\dHopf}$ is a symmetric monoidal category.
\end{prop}
\begin{proof}
	This follows from the standard arguments on stable module category using the fact that the tensor product preserves the projectives.
\end{proof}

\subsection{Projective and injective objects}
The following lemmas tell us what are the projective/injective objects in the exact category we are interested in.
\begin{lem}\label{lem:Aproj-Rproj}
	If $M$ is a finitely generated projective $\dHopf$-module, then it is a finitely generated projective $R$-module.
\end{lem}
\begin{proof}
	A finitely generated projective $\dHopf$-module $M$ is a summand of a finitely generated free $\dHopf$-module, which is finitely generated and projective over $R$ since $\dHopf$ itself is so. Therefore,  $M$ is a finitely generated projective $R$-module.
\end{proof}
In particular, all finitely generated projective $\dHopf$-modules are contained in the subcategory $\Hmodp$ and coincide with its projective objects. 
Similar statement for injective objects fails. 
However, we have
\begin{lem}\label{lem:Ainj}
	If $M$ is an injective object in $\Hcomodp$, then it is an injective $\Hopf$-comodule. 
\end{lem}
\begin{proof}
	An $\Hopf$-comodule $M$ is injective if and only if the functor $\Hom_{\Hopf}[-,M]$ is exact on the abelian category of $\Hopf$-comodules. 
	But \cite{hovey2003homotopy}*{Proposition 1.4.4} tells us that $\Hcomodp$ generates this category. 
	Hence, it suffices to verify if $\Hom_{\Hopf}[-,M]$ is exact on the exact category $\Hcomodp$, which is precisely the definition of injective objects in $\Hcomodp$.
\end{proof}

\subsection{(Stable) Picard group}
The \emph{Picard group} of a symmetric monoidal category is the group of invertible isomorphism classes.
\begin{defn}\label{defn:Pic}
	The \emph{(stable) Picard group} of $\Hopf$ (or of $\dHopf$) is the Picard group of the symmetric monoidal category $\stab{\dHopf}$.
	We will use $\Pic(\Hopf)$ or $\Pic(\dHopf)$ to denote this group.
\end{defn}
Two $\dHopf$-modules $M$, $N$ are \emph{stably equivalent} if $M\oplus P\cong N\oplus Q$ for some projectives $P$ and $Q$. If $M\tensor{R}N$ is stably equivalent to $R$, then we say $M$ and $N$ are \emph{stably invertible}. 
Clearly, two objects in $\Hmodp$ are stably equivalent if and only if they are isomorphic in $\stab{\dHopf}$.
In particular, an object $M$ in $\stab{\dHopf}$ is invertible if and only if it is stably invertible.

\begin{remark}
	In the \cref{defn:Pic}, we specifically define the Picard group using objects in $\Hmodp$. 
	To justify this choice, let's consider $\dHopf$-modules $M$ and $N$ such that 
	$M\tensor{R}N$ is stably equivalent to $R$. 
	In this case, $M\tensor{R}N$ becomes a nonzero finitely generated projective $\dHopf$-module, and hence, by \cref{lem:Aproj-Rproj}, a nonzero finitely generated projective $R$-module. 
	Then the factors $M$ and $N$ have to be nonzero finitely generated projective $R$-modules.
\end{remark}

The rest of this section devotes to give the ordinary part of the Picard group from a study of the category $\stab{\dHopf}$.

\subsection{Dual (co)modules}\label{dualization}
Let's first show that
\begin{prop}\label{prop:RigidityOfHmodp}
	The symmetric monoidal category $\Hmodp$ is \emph{left rigid}: any object in it admits a left dual.
\end{prop}

For an $\Hopf$-comodule $M$, its $R$-dual $\dual{M}$ naturally carries an $\dHopf$-module structure:
\[
	\dHopf\tensor{R}\dual{M}
	\longrightarrow\dual{M}\colon
	f\tensor{} g
	\longmapsto
	\braket*{f}{-_{[-1]}\cdot\eta\braket{g}{-_{[0]}}}.
\]
Conversely, let $(M,\rho_M)$ be an $\dHopf$-module. Then the canonical map 
\[
	\Hopf\tensor{R}\dual{M}
		\longrightarrow\dual{(\dHopf\tensor{R}M)}\colon
	x\tensor{} g 
	\longmapsto
	\braket*{-}{x\cdot\eta\braket*{g}{-}}
\]
is an isomorphism if and only if $M$ is finitely generated and projective as an $R$-module.
If this is the case, the dual $\dual{M}$ carries an $\Hopf$-comodule structure 
$\psi_{\dual{M}}\colon\dual{M}\to\Hopf\tensor{R}\dual{M}$ 
given by composing $\dual{\rho_M}$ with the inverse of the above map.
One can then verify that the above constructions form a monoidal anti-equivalence between the exact categories $\Hcomodp$ and $\Hmodp$. 

\begin{defn}
	Let $D$ denote the $R$-linear functor 
	\[
		D\colon \opp{\Hmodp} \longrightarrow \Hmodp
	\]
	obtained as the composition of $\dual{}$ and $c$.
\end{defn}

\begin{proof}[Proof of \cref{prop:RigidityOfHmodp}]
	Recall that for a finitely generated projective $R$-module $M$, its left dual is given by the weak dual $\dual{M}$. 
	Then the discussion in \ref{comodule=module} and the above show that any object $M$ in $\Hmodp$ admits a left dual given by $DM$.
\end{proof}

\subsection{Extended (co)modules}
To induce a functor on $\stab{\dHopf}$ from $D$, we need the following notions.
\begin{defn}
	Given an $R$-module $M$, the tensor product $\dHopf\tensor{R}M$ (resp. $\Hopf\tensor{R}M$) naturally carries an $\dHopf$-module (resp. $\Hopf$-comodule) structure using the (co)multiplication of $\dHopf$ (resp. $\Hopf$). 
	This (co)module is called the \emph{extended (co)module} on $M$.
\end{defn}

If $M$ is already an $\dHopf$-module (resp. $\Hopf$-comodule), then we have \emph{a priori} two $\dHopf$-module (resp. $\Hopf$-comodule) structures on $\dHopf\tensor{R}M$ (resp. $\Hopf\tensor{R}M$): one from the monoidal structure of $\Hmod$ (resp. $\Hcomod$), and another from the extended (co)module construction. 
It turns out that 
\begin{lem}
	The two (co)module structures on $\dHopf\tensor{R}M$ (resp. $\Hopf\tensor{R}M$) are isomorphic.
\end{lem}
\begin{proof}
	It is sufficient to prove this for $\Hopf$-comodules.
	Which follows by considering the bijections
	\[
		x\tensor{} m\longmapsto x_{(1)}m_{[-1]}\tensor{} \epsilon(x_{(2)}).m_{[0]},\qquad
		x\tensor{} m\longmapsto
		x\cdot c(m_{[-1]})\tensor{} m_{[0]}
	\]
	as in \cite{hovey2003homotopy}*{Lemma 1.1.5}.
\end{proof}
\begin{cor}\label{cor:Ext->Free}
	If $F$ is a free $R$-module of finite rank, then the tensor product $\dHopf\tensor{R}F$ (resp. $\Hopf\tensor{R}F$) is a (co)free $\dHopf$-module (resp. $\Hopf$-comodule).
\end{cor}
\begin{proof}
	It is clear that the extended (co)module $\dHopf\tensor{R}F$ (resp. $\Hopf\tensor{R}F$) is (co)free. Then the statement follows.
\end{proof}

\subsection{Rigidity of $\stab{\dHopf}$}
We are now able to prove the following:
\begin{thm}\label{thm:DH-proj}
	Suppose $\Hopf$ is free over $R$. 
	Then $D\dHopf$ is a projective $\dHopf$-module.
\end{thm}
\begin{proof}
		Since $c\colon\Hopf\to\Hopf$ is an anti-isomorphism between the left and right $R$-module structures on $\Hopf$, 
		the assumption that $(\Hopf,\iota)$ is a free left $R$-module implies that $(D\dHopf,r)=(\Hopf,\eta)$ is a free right $R$-module. 
		Following \cite{margolis}*{Theorem 12.2.9}, let $\{x_i\}$ be an $R$-basis of this right $R$-module containing $1$. 
	Then we have a retraction:
	\[
		\dHopf\longmono
		D\dHopf\tensor{R}\dHopf\longepi
		\dHopf
	\]
	given by $1\mapsto 1\tensor{} 1$ and $1\tensor{} 1\mapsto 1$ (the image of other $x_i\tensor{} 1$ doesn't matter). 
	Applying the functor $D$ to it, we obtain a retraction
	\[
		D\dHopf\longmono
		\dHopf\tensor{R}D\dHopf\longepi
		D\dHopf,
	\]
	where $\dHopf\tensor{R}D\dHopf$ is a free $\dHopf$-module by \cref{cor:Ext->Free}. 
	Therefore, $D\dHopf$ is a projective $\dHopf$-module.
\end{proof}

In what follows, $\Hopf$ is assumed to be free as an $R$-module. 
In our application, the ground ring $R$ is a polynomial ring. 
Then the \emph{Quillen-Suslin Theorem} \cite{Lang}*{XXI, Theorem 3.7} tells us that all projective $R$-modules are free over $R$. 
Hence, all finite Adams Hopf algebroids $(\Hopf,R)$ with $R$ a polynomial ring satisfy the assumption. Moreover, the conclusion of \cref{lem:Aproj-Rproj} is then improved to ``free over $R$''.
\begin{cor}\label{cor:dualizationOnStab}
	The $R$-linear functor $D\colon\opp{\Hmodp}\to\Hmodp$ induces a functor
	\[
		D\colon \opp{\stab{\dHopf}} \longrightarrow \stab{\dHopf}.
	\]
\end{cor}
\begin{proof}
	We need to show $D$ maps projectives to projectives. Since $D$ commutes with direct sums, the statement follows from \cref{thm:DH-proj}. 
\end{proof}
\begin{cor}\label{Frobenius}
	The exact category $\Hmodp$ is a \emph{Frobenius category}. That is to say, it has enough projectives and injectives, and they coincide.
\end{cor}
\begin{proof}
	Any object $M$ in $\Hmodp$ admits a projective covering $\dHopf\tensor{R}M \to M$ from the extended module. 
		This shows that $\Hmodp$ has enough projectives.
	The dualization $D$ maps any projective covering to an injective hull. 
	This shows that $\Hmodp$ has enough injectives.
	Finally, \cref{thm:DH-proj} tells us that $D$ maps projectives to projectives. Therefore, projectives and injectives in $\Hmodp$ coincide.
\end{proof}

\begin{cor}
	The symmetric monoidal category $\stab{\dHopf}$ is left rigid.
\end{cor}
\begin{proof}
	This follows from \cref{prop:RigidityOfHmodp} and \cref{cor:dualizationOnStab}.
\end{proof}

The following lemma gives us a criterion of stable invertibility.
\begin{lem}\label{invertibility}
	Let $M$ be an object in $\Hmodp$. 
	The evaluation morphism 
	\[DM\tensor{R}M\xrightarrow{eval}R\] 
	is a stable equivalence if and only if $M$ is invertible in $\stab{\dHopf}$. 
	In particular, the inverse of $[M]$ in $\Pic(\dHopf)$ is given by its dual $[DM]$.
\end{lem}
\begin{proof}
	This is a standard result in a rigid monoidal category.
\end{proof}
\begin{remark}
	Since $DM\tensor{R}M\iso\End_{R}[M]$, objects verifying \cref{invertibility} are said to be \emph{endo-trivial}.
\end{remark}

\subsection{Algebraic loop functor}
The category $\stab{\dHopf}$ has the following extra structure.
\begin{defn}\label{def:algloop}
	Let $\Omega$ be the endo-functor of $\stab{\dHopf}$ given by 
	\[
		\Omega M = \ker(PM\longrightarrow M),
	\]
	where $PM\to M$ is any projective cover of $M$. For $k>0$, define $\Omega^kM$ inductively to be $\Omega (\Omega^{k-1} M)$; for $k<0$, define $\Omega^k M$ to be $D\Omega^{-k}DM$.
\end{defn}
\begin{remark}
	Note that $\Omega M$ is a finitely generated projective $R$-module since it is a submodule of $P$, which is a finitely generated and projective over $R$ by \cref{lem:Aproj-Rproj}.
\end{remark}

An immediate application of the following Schanuel's lemma shows that, up to stable equivalence, $\Omega M$ is independent of the choice of $P$.

\begin{lem}[Schanuel's lemma \cite{Lang}*{XXI, Lemma 2.4}]\label{SchanuelLemma}
	Suppose we are given the solid arrows in the following diagram, where the horizontal lines are exact sequences and $P,P'$ are projectives. 
	\[\begin{tikzcd}
	0 & {K} & {P} & {M} & 0 \\
	0 & {K'} & {P'} & {M} & 0
	\arrow[from=1-1, to=1-2]
	\arrow[from=1-2, to=1-3]
	\arrow[from=1-3, to=1-4]
	\arrow[from=1-4, to=1-5]
	\arrow[from=2-1, to=2-2]
	\arrow[from=2-2, to=2-3]
	\arrow[from=2-3, to=2-4]
	\arrow[from=2-4, to=2-5]
	\arrow[dotted, from=1-2, to=2-2]
	\arrow[dotted, from=1-3, to=2-3]
	\arrow[Rightarrow, no head, from=1-4, to=2-4]
	\end{tikzcd}\]
	Then the commutative diagram is completed by the dotted arrows making the left square a Cartesian square. In particularly, $K\oplus P'\cong K'\oplus P$.
\end{lem}

\begin{defn}
	The tensor unit $R$ of $\Hmod$ admits a projective covering $\dual{\iota}\colon\dHopf\to R$ given by its module structure. Its kernel is called the \emph{augmentation ideal} $I$ of $\dHopf$. 
	Dually, the cokernel of $\eta\colon R\to D\dHopf$ is denoted by $I^{-1}$.
\end{defn}
\begin{lem}\label{Iisinj}
	The augmentation ideal $I$ is stably invertible.
\end{lem}
\begin{proof}
	In fact, $I^{-1}=DI$ is its inverse. To see this, we can tensor the short exact sequence $0\to I\to\dHopf\to R\to 0$ with $I^{-1}$ and get
	\[
		0\longrightarrow I\tensor{R}I^{-1}\longrightarrow\dHopf\tensor{R}I^{-1}\to I^{-1}\longrightarrow 0.
	\]
	Note that the middle term is projective: indeed, since $I^{-1}$ is finitely generated and projective as an $R$-module, applying \cref{cor:Ext->Free} to a finite free $R$-module containing $I^{-1}$ as its summand, we see that $\dHopf\tensor{R}I^{-1}$ is projective. 
	Comparing the above sequence with $0\to R\to D\dHopf\to I^{-1}\to 0$, \cref{SchanuelLemma} tells us that $I\tensor{R}I^{-1}$ is stably equivalent to $R$. 
\end{proof}
\begin{cor}
	If $M$ is stably invertible, then so is $\Omega M$. 
\end{cor}
\begin{proof}
	It suffices to show that $\Omega M$ is stably equivalent to $I\tensor{R}M$, which follows from the uniqueness of $\Omega M$.
\end{proof}
\begin{thm}\label{TrivialsOfPic}
	Suppose our working category $\K$ is graded by an abelian group $\Gamma$. Then we have an injection
	\[
		\Gamma\times\Z\longmono \Pic(\dHopf),
	\]
	where the first component comes from the grading and the second component comes from $I$.
\end{thm}
\begin{proof}
	Let $\Sigma^{\gamma}$ be the grading shifting functor corresponding to $\gamma\in\Gamma$. 
		Then $\Sigma^{\gamma}R$ is stably invertible with inverse $\Sigma^{-\gamma}R$. 
		This explains the first component. 
		The second component follows from \cref{Iisinj}.
\end{proof}
\begin{defn}
	We call the image of the injection in \cref{TrivialsOfPic} the \emph{ordinary part} of the Picard group $\Pic(\dHopf)$ since it is always contained in the Picard group.
\end{defn}

\subsection{Triangulated structure}
With above results, we can verify that $\stab{\dHopf}$ is a tensor triangulated category whose suspension is given by the functor $\Omega^{-1}$ and whose distinguish triangles $M_1\to M_2\to M_3\to \Omega^{-1}M_1$ are given by the short exact sequences $0\to M_1\to M_2\to M_3\to 0$ of $\dHopf$-modules.

We finish this section with a discussion on the Ext groups.
\begin{lem}
	For any objects $M,N$ in $\Hmodp$ and $k\geqslant 0$, we have 
	\[
		\Hom_{\stab{\dHopf}}[\Omega^kM,N]\cong \Ext_{\dHopf}^k[M,N].
	\]
\end{lem}
\begin{proof}
	Note that $\Ext_{\dHopf}^{\ast}[-,N]$ is the right derived functor of the left exact functor $\Hom_{\dHopf}[-,N]$ on $\Hmod$ and is hence computed by projective resolutions. But by lemma \ref{lem:Aproj-Rproj}, all projectives in $\Hmod$ are in $\Hmodp$. 
	In particular, any object $M$ in $\Hmodp$ admits a projective resolution 
	\[
		\cdots\longrightarrow P\Omega^kM\longrightarrow\cdots\longrightarrow P\Omega M\longrightarrow PM\longrightarrow M.
	\]
	From which the statement follows as in \cite{margolis}*{Proposition 14.1.8}.
\end{proof}

\section{The \texorpdfstring{$\R$}{R}-motivic scenario}\label{sec:mot}
The discussion in the previous section can be applied to a general setting. 
Yet the goal of this paper is rather humble: to compute the (stable) Picard group of $\cAR$, which is the subalgebra of the \emph{mod $2$ $\R$-motivic Steenrod algebra} $\cA^\R$ generated by $\Sq^1$ and $\Sq^2$.

\subsection{\texorpdfstring{$\R$}{R}-motivic Steenrod algebra}
Our working category $\K$ is the following one:
\begin{description}[wide]
	\item[{$\cat{Vect}[\F_2]$}] the category of \emph{bounded-below bigraded vector spaces} over $\F_2$. 
	Its objects are bigraded $\F_2$-vector spaces 
	$V^{\bullet,\bullet}=\bigoplus_{s,w}V^{s,w}$ with $V^{s,w}$ vanishes for either $s\ll 0$ or $w\ll 0$. 
	Its monoidal structure is given by the graded tensor product, and the symmetry is given by $a\tensor{} b\mapsto(-1)^{\abs{a}\abs{b}}b\tensor{} a$. 
	To keep notations concise, we will omit the sign $(-1)^{\abs{a}\abs{b}}$. 
\end{description}

Our base commutative $\K$-algebra $R^{\bullet,\bullet}$ is given by the \emph{$\R$-motivic cohomology over a point}, denoted by $\MtwoR$. 
Where the first grading is the \emph{classical internal degree} and the second one is the \emph{motivic weight}.
It may be convenient to employ another working category $\cat{Mod}[\MtwoC]$, where $\MtwoC$ is the \emph{$\C$-motivic cohomology over a point}. 
Then $\MtwoR$ can also be viewed as an $\MtwoC$-algebra.

The explicit descriptions of $\MtwoR$ and $\MtwoC$ follow from Voevodsky's proof of the \emph{Milnor conjecture} in \cite{voevodsky2003motivic}*{Corollary 6.9(2) and Corollary 6.10}:
\begin{thm}
	As $\K$-algebras, 
	$\MtwoR\iso \F_2[\tau,\rho]$, 
	and 
	$\MtwoC\iso \F_2[\tau]$, 
	where $\tau$ has degree $(0,1)$ and $\rho$ has degree $(1,1)$.
\end{thm}
\begin{remark}
	It is worth noting that $\MtwoR$ and $\MtwoC$ are local ring objects in $\K$, namely \emph{graded local rings}. 
	Indeed they are \emph{connected} bigraded $\F_2$-algebras. 
	Hence, they have only one maximal homogeneous ideal. 
	As a consequence, any projectives in $\cat{Mod}[\MtwoR]$ (resp. $\cat{Mod}[\MtwoC]$) are automatically free.
\end{remark}

According to a computation of Voevodsky in \cite{voevodsky2003reduced}*{\S 11 and 12}, the \emph{$\R$-motivic Steenrod algebra} $\cA^\R$ can be regarded as an associative algebra over $\MtwoR$, 
generated by the \emph{Steenrod squares} $\Sq^{2i}$ and $\Sq^{2i-1}$ in bidegrees $(2i,i)$ and $(2i-1,i-1)$, respectively, subject to the \emph{$\R$-motivic Adem relations} outlined in \cite{voevodsky2003reduced}*{Theorem 10.2}. 
The whole algebra $\cA^\R$ is not finitely generated as an $\MtwoR$-module, while the subalgebra $\cAR$ has a basis with $8$ generators (see  \cref{fig:motivicA1}). 

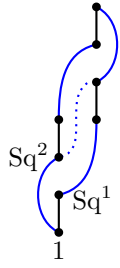
\begin{figure}[ht]
    \centering
    \begin{tikzpicture}\begin{scope}[thick, every node/.style={sloped,allow upside down}, scale=0.5]
    \draw (0,0)  node[inner sep=0] (v00) {} -- (0,1) node[inner sep=0] (v01) {};
    \draw (0,2)  node[inner sep=0] (v11) {} -- (0,3) node[inner sep=0] (v12) {};
    \draw (1,3)  node[inner sep=0] (v22) {} -- (1,4) node[inner sep=0] (v23) {};
    \draw (1,5)  node[inner sep=0] (v33) {} -- (1,6) node[inner sep=0] (v34) {};
    \draw [color=blue] (v00) to [out=150,in=-150] (v11);
    \draw [color=blue] (v01) to [out=15,in=-90] (v22);
    \draw [color=blue] (v12) to [out=90,in=-165] (v33);
    \draw [color=blue] (v23) to [out=30,in=-30] (v34);
    \draw [dotted][color=blue] (v11) to [out=15,in=-150] (.8, 4);
    \filldraw (v00) circle (2.5pt);
    \filldraw (v01) circle (2.5pt);
    \filldraw (v11) circle (2.5pt);
    \filldraw (v12) circle (2.5pt);
    \filldraw (v22) circle (2.5pt);
    \filldraw (v23) circle (2.5pt);
    \filldraw (v33) circle (2.5pt);
    \filldraw (v34) circle (2.5pt);
    \node[below] at (v00) {$1$};
    \node[right,xshift={1.5pt},yshift={-2pt}] at (v01) {$\Sq^1$};
    \node[left] at (v11) {$\Sq^2$};
    \end{scope}\end{tikzpicture}
    \caption{
    We depict $\cAR$ as its free module over $1$: 
		each $\bullet$ represents a $\MtwoR$-generator; 
		the black and blue lines represent the action of $\Sq^1$ and $\Sq^2$, reading from bottom to top, respectively; and 
		a line is dotted means that the action hits the $\tau$-multiple of the given $\MtwoR$-generator. 
    }
    \label{fig:motivicA1}
\end{figure}

Unlike the \emph{$\C$-motivic Steenrod algebra} $\cA^\C$, 
the $\R$-motivic Steenrod algebra $\cA^\R$ does not possess the structure of a Hopf algebra. 
However, its $\MtwoR$-dual, denoted as $\dual{\cA}_\R$, carries a Hopf algebroid structure with $\MtwoR$ as its base. 
The same happens to the subalgebra $\cAR$: its $\MtwoR$-dual, referred to as $\dcAR$, carries the structure of a finite Adams Hopf algebroid structure with $\MtwoR$ as its base. 
Since $\MtwoR$ is a polynomial ring, the stable module category $\stab{\cAR}$ is a rigid category as showed in \cref{sec:background}. 
group.
In particular, by \cref{TrivialsOfPic}, we have the following:
\begin{lem}\label{lem:PicZ3}
    We have an injection of groups
    \[
        \Z^3\longmono \Pic(\cAR),
    \]
    mapping $(s,w,k)\in\Z^3$ to the class $[\Sigma^{s,w}\Omega^k\MtwoR]$, where $\Sigma^{s,w}$ is the grading shifting functor of cohomological degree $(s,w)$ and $\Omega^k$ is the $k$-th algebraic loop functor defined in \cref{def:algloop}.
\end{lem}

More concretely, the Picard group $\Pic(\cAR)$ contains at least a free abelian group generated by the isomorphism classes of the following three invertible objects in $\stab{\cAR}$:
\begin{enumerate}
    \item The $\cAR$-module $\Sigma_s:=\Sigma^{1,0}\MtwoR$, that is $\MtwoR$ in degree $(1,0)$ and zero in other degrees. 
    \item The $\cAR$-module $\Sigma_w:=\Sigma^{0,1}\MtwoR$, that is $\MtwoR$ in degree $(0,1)$ and zero in other degrees. 
    \item The augmentation ideal $I$.
\end{enumerate}

Following \cite{adams1976uniqueness}, we should in addition consider the following exotic element:
\begin{enumerate}[resume]
    \item The $\cAR$-module $J:=\cAR/\cAR\Sq^3$ (called the \emph{joker}), depicted as in \cref{fig:joker}. 
		Its inverse is given by switching the dotted line to the bottom.
    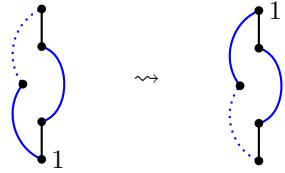
\begin{figure}[!h]
        \[
        \begin{tikzpicture}[baseline={([yshift=.5ex]current bounding box.center)}]
        \begin{scope}[ thick, every node/.style={sloped,allow upside down}, scale=0.5]
        \draw (3,0)  node[inner sep=0] (v30) {} -- (3,1) node[inner sep=0] (v31) {};
        \draw  (2.5,2) node[inner sep=0] (v22) {};
        \draw (3,3)  node[inner sep=0] (v33) {} -- (3,4) node[inner sep=0] (v34) {};
        \draw [color=blue] (v31) to [bend right=70] (v33);
        \draw [color=blue] (v30) to [bend left = 50] (v22);
        \draw[dotted] [color=blue] (v22) to [bend left =50] (v34);
        \filldraw (v30) circle (2.5pt);
        \filldraw (v31) circle (2.5pt);
        \filldraw (v22) circle (2.5pt);
        \filldraw (v33) circle (2.5pt);
        \filldraw (v34) circle (2.5pt);
        \node[right] at (v30) {$1$};
        \end{scope}\end{tikzpicture}
        \qquad\leadsto\qquad
        \begin{tikzpicture}[baseline={([yshift=-.5ex]current bounding box.center)}]
        \begin{scope}[ thick, every node/.style={sloped,allow upside down}, scale=0.5, baseline=1cm]
        \draw (3,0)  node[inner sep=0] (v30) {} -- (3,1) node[inner sep=0] (v31) {};
        \draw  (2.5,2) node[inner sep=0] (v22) {};
        \draw (3,3)  node[inner sep=0] (v33) {} -- (3,4) node[inner sep=0] (v34) {};
        \draw [color=blue] (v31) to [bend right=70] (v33);
        \draw[dotted] [color=blue] (v30) to [bend left = 50] (v22);
        \draw [color=blue] (v22) to [bend left =50] (v34);
        \filldraw (v30) circle (2.5pt);
        \filldraw (v31) circle (2.5pt);
        \filldraw (v22) circle (2.5pt);
        \filldraw (v33) circle (2.5pt);
        \filldraw (v34) circle (2.5pt);
        \node[right] at (v34) {$1$};
        \end{scope}\end{tikzpicture}
        \]
        \caption{The $\cAR$-module $J$ and its inverse.} 
        \label{fig:joker}
    \end{figure}
\end{enumerate}

In order to manage the invertibility of $J$, we need to introduce more useful tools.

\subsection{$\rho$ and $\tau$ quotients}\label{sec:quo}

Adding the condition of being free over the coefficient rings, we obtain the following definition.
\begin{defn}
A \emph{finite $\R$-motivic Hopf algebra} $A$ is the $\MtwoR$-dual of a finite Adams Hopf algebroid with base $\MtwoR$. 
In particular, $A$ is finitely generated and free as an $\MtwoR$-module.
A \emph{finite $\C$-motivic Hopf algebra} is a cocommutative bigraded Hopf algebra over $\MtwoC$ that is finitely generated and free as an $\MtwoC$-module. Finally, a \emph{finite classical Hopf algebra} is a finite-dimensional cocommutative bigraded Hopf algebra over $\F_2$.
\end{defn}
\begin{remark}
    Note that A finite $\R$-motivic Hopf algebra is in fact NOT a Hopf algebra over $\MtwoR$.
\end{remark}

Suppose $A$ is a finite $\R$-motivic Hopf algebra. 
Its base change from $\MtwoR$ to $\MtwoC$ gives a finite $\C$-motivic Hopf algebra $A/\rho:= \MtwoC\tensor{\MtwoR}A$. Going one step further, the base change to $\F_2$ gives a finite classical Hopf algebra $A/(\rho, \tau):= \F_2\tensor{\MtwoR}A$. 
Because the finite classical Hopf algebras are better analyzed, we will relate the Picard group of a finite $\R$-motivic Hopf algebra $A$ to the Picard group of its base change $A/(\rho, \tau)$.

Following the same strategy in \cite{gheorghe2018picard}*{\S 3}, we have the followings:
\begin{thm}\label{basechange}
Let $A$ be a finite $\R$-motivic Hopf algebra. 
The base changes from $\MtwoR$ to $\MtwoC$ and to $\F_2$ induce strongly monoidal functors 
\[
	(-)/(\rho,\tau)\colon
	\modp{A}\xrightarrow{(-)/\rho}
	\modp{A/\rho}\xrightarrow{(-)/\tau}
	\modp{A/(\rho,\tau)}
\]
that preserve exact sequences. 
These functors pass to the stable module categories and thus 
induce strongly monoidal triangulated functors
\[
	(-)/(\rho,\tau)\colon
	\stab{A}\xrightarrow{(-)/\rho}
	\stab{A/\rho}\xrightarrow{(-)/\tau}
	\stab{A/(\rho,\tau)}.
\] 
In particular, for any object $M$ in $\modp{A}$, we have 
\[
	DM/(\rho,\tau)\iso 
	D(M/(\rho,\tau)).
\]   
\end{thm}
\begin{proof}
	The strong monoidalities follow from standard results on base changes. 
	The exactness follow from the fact that short exact sequences split in these categories.
	Finally, the functors pass to the stable module categories since they preserve freeness and exact sequences.
\end{proof}

We can relate the projectivity over $A$ with its base change.

\begin{cor}\label{proj=free}
    Let $A$ be a finite $\R$-motivic Hopf algebra. 
		Then, for any object $M$ in $\modp{A}$, the following are equivalent:
    \begin{enumerate}[label=\textup{(\arabic*)}]
        \item $M$ is projective as an $A$-module;
        \item $M/\rho$ is projective as an $A/\rho$-module;
        \item $M/(\rho,\tau)$ is projective as an $A/(\rho,\tau)$-module;
        \item $M/(\rho,\tau)$ is free as an $A/(\rho,\tau)$-module;
        \item $M/\rho$ is free as an $A/\rho$-module;
        \item $M$ is free as an $A$-module.
    \end{enumerate}
\end{cor}

\begin{proof}

Note that: 
(1)$\Rightarrow$(2)$\Rightarrow$(3) follow from \cref{basechange}, and that (3)$\Rightarrow$(4) is a consequence of the fact that $\MtwoR$ is a graded local ring. 
The proofs of (4)$\Rightarrow$(5) and (5)$\Rightarrow$(6) are essentially the same. 
Here we will only present the latter.

(5)$\Rightarrow$(6). 
Given a basis of the free $A/\rho$-module $M/\rho$, we can lift it to a free $A$-module $F$ with a map $f\colon F\to M$. 
By assumption, $f/\rho$ is an epimorphism. 
Hence, by the \emph{Graded Nakayama Lemma}, $f$ is also an epimorphism. 
Consider the following exact sequence. 
\[\begin{tikzcd}
0 & {\Ker[f]} & {F} & {M} & 0
\arrow[from=1-1, to=1-2]
\arrow[from=1-2, to=1-3]
\arrow[from=1-3, to=1-4,"{f}"]
\arrow[from=1-4, to=1-5]
\end{tikzcd}\]

Since $M$ is projective over $\MtwoR$, the exact sequence splits, and thus $\Ker[f]$ is a retraction of $F$, which is a free $\MtwoR$-module. 
Hence, $M$ is also projective over $\MtwoR$, and thus the above short exact sequence lives in the exact category $\modp{A}$. 
By \cref{basechange}, its base change to $\MtwoC$ is again exact. By assumption, $f/\rho$ is an isomorphism. 
Hence, $\Ker[f]/\rho=\Ker[f/\rho]=0$. 
Applying the \emph{Graded Nakayama Lemma}, we see that $\Ker[f]$ vanishes. 
Therefore, $f$ is an isomorphism.
\end{proof}

Note that, although \cref{basechange} tells us that the base changes preserve exact sequences, 
it does not thus imply that the base changes reflect 
monomorphisms since an exact subcategory needs not to 
be closed under taking cokernel.
However, the following lemmas show that 
the base changes do reflect monomorphisms. 
\begin{lem}\label{lem:ReflectMonosR}
	Let $A$ be a finite $\R$-motivic Hopf algebra. 
	If a morphism $f\colon M\to N$ in $\modp{A}$ induces a monomorphism  $M/\rho\to N/\rho$, then $f$ itself is also monic. 
	Moreover, the cokernel of $f$ is also in $\modp{A}$.
\end{lem}
\begin{lem}\label{lem:ReflectMonosC}
	Let $A$ be a finite $\C$-motivic Hopf algebra. 
	If a morphism $f\colon M\to N$ in $\modp{A}$ induces a monomorphism  $M/\tau\to N/\tau$, then $f$ itself is also monic. 
	Moreover, the cokernel of $f$ is also in $\modp{A}$.
\end{lem}
The two lemmas both follow from the following technique lemma (notice that, finitely presented $+$ flat $\implies$ free).
\begin{lem}
	Let $(R,\mm)$ be a noetherian local ring object in a locally noetherian symmetric monoidal abelian category $\K$. 
	If a morphism $f\colon M\to N$ of finitely generated flat $R$-modules induces a monomorphism $M/\mm M\to N/\mm N$, then $f$ itself is also monic. 
	Moreover, the cokernel of $f$ is also flat.
\end{lem}

\begin{proof}
	For any $n$, let $f_n$ denote the base change $f\tensor{R}R/\mm^n$. 
	We show that $f_n$ are monomorphisms by induction. 
	First, $f_1$ is monic by assumption. 
	For the induction, consider the following commutative diagram:
	\[\begin{tikzcd}
	0 & {M\tensor{R}\mm^n/\mm^{n+1}} & {M/\mm^{n+1}M} & {M/\mm^{n}M} & 0 \\
	0 & {N\tensor{R}\mm^n/\mm^{n+1}} & {N/\mm^{n+1}N} & {N/\mm^{n}N} & 0
	\arrow[from=1-1, to=1-2]
	\arrow[from=1-2, to=1-3]
	\arrow[from=1-3, to=1-4]
	\arrow[from=1-4, to=1-5]
	\arrow[from=2-1, to=2-2]
	\arrow[from=2-2, to=2-3]
	\arrow[from=2-3, to=2-4]
	\arrow[from=2-4, to=2-5]
	\arrow[from=1-2, to=2-2,"{f\tensor{R}\mm^n/\mm^{n+1}}"]
	\arrow[from=1-3, to=2-3,"{f_{n+1}}"]
	\arrow[from=1-4, to=2-4,"{f_{n}}"]
	\end{tikzcd}\]
	where the horizontals are exact since both $M$ and $N$ are flat. 
	The first vertical arrow $f\tensor{R}\mm^n/\mm^{n+1}\iso f_1$ is monic since $\mm^n/\mm^{n+1}\iso R/\mm$ identifies it with $f_1$. 
	The last vertical arrow $f_{n}$ is monic by inductive hypothesis. 
	Therefore, $f_{n+1}$ is monic by the \emph{Five Lemma}.

	For each $n$, since $f_{n}$ is monic, the composition $\Ker[f]\to M\to M/\mm^n M$ is zero. 
	Hence, $\Ker[f]$ is a subobject of $\mm^n M$. 
	Since $M$ is finitely generated over $R$, the limit $\bigcap_n\mm^n M$ vanishes. 
	Hence, $\Ker[f]=0$ and $f$ is monic. 

	To show $\Coker[f]$ is flat, it suffices to show $\Tor^{1}[\Coker[f],R/I]=0$ for all ideals $I$ of $R$. 
	Since $M$ is flat, we have the following exact sequence.
	\[\begin{tikzcd}
	0 & {\Tor^{1}[\Coker[f],R/I]} & {M/IM} & {N/IN} 
	\arrow[from=1-1, to=1-2]
	\arrow[from=1-2, to=1-3]
	\arrow[from=1-3, to=1-4,"{f/I}"]
	\end{tikzcd}\]
	Hence, to show $\Tor^{1}[\Coker[f],R/I]=0$, it suffices to show $f/I$ is monic. 
	Note that both $M/IM$ and $N/IN$ are finitely generated flat over $R/I$, and that $f/I$ is already monic modulo $\mm$. 
	Hence, applying the first part of the lemma to the local ring $(R/I,\mm/I\mm)$ and the morphism $f/I$, we conclude that $f/I$ is monic.
\end{proof}

\begin{prop}\label{injPic}
	Let $A$ be a finite $\R$-motivic Hopf algebra. 
	The base changes induce homomorphisms between the Picard groups
	\[ 
		\Pic(A)\longrightarrow
		\Pic(A/\rho)\longrightarrow
		\Pic(A/(\rho,\tau)).
	\] 
	Furthermore, they are injective.
\end{prop}

\begin{proof}
	The homomorphisms come from \cref{basechange}. 
	Suppose $M$ is an object in $\modp{A}$ such that $M/\rho$ is stably equivalent to $\MtwoC$ as $A/\rho$-modules. 
	Then following the same strategy in \cite{margolis}*{Proposition 14.11}, we can express $M/\rho$ as the direct sum of $\MtwoC$ with a free $A/\rho$-module. 
	We may lift this free $A/\rho$-module to a free $A$-module $F$ with a morphism $f\colon F\to M$ such that $f/\rho$ gives the inclusion $F/\rho\hookrightarrow M/\rho$. 
	Then, by \cref{lem:ReflectMonosR}, we have the following short exact sequence in $\modp{A}$:
	\[
		F\overset{f}{\longmono} M\longepi \Coker[f].
	\]
	Since $\Coker[f]/\rho = \Coker[f/\rho] = \MtwoC$, we conclude that $\Coker[f]\iso\MtwoR$, and thus the class $[M]$ is trivial in $\Pic(A)$.
\end{proof}

\subsection{Margolis homology and a freeness criterion}
Margolis homology is a tool to detect freeness in the classical setting. 
Suppose $a$ is an element in $A$ that squares to zero. For any $A$-module $M$, the Margolis homology $H(M,a)$ is the annihilator of $a$ modulo the submodule $aM$.

We have built strong connections between finitely generated $\MtwoR$-free modules over a finite $\R$-motivic Hopf algebra and their base changes to $\F_2$. 
As a result, we can use the Margolis homology to detect the freeness of the module. 
Classically, $M$ is $\cA(1)$-free if and only if the Margolis homology $H(M,Q_0)$ and $H(M,Q_1)$ vanish. For the $\C$ and $\R$-motivic context, we need to further take $H(M,\Sq^2)$ into account.
\begin{thm}[\citelist{\cite{gheorghe2018picard}*{Propersition 4.7}\cite{bhattacharya2022}*{Corollary 2.4}}]\label{thm:Margolis}
	Let $M$ be an object in the subcategory $\modp{\cAR}$. 
	Then $M$ is free over $\cAR$ if and only if 
	\[
		H(M/(\rho,\tau);\alpha)=0
	\]
	for all $\alpha\in\Set*{Q_0,Q_1,\Sq^2}$.
\end{thm}

\begin{cor}\label{cor:Margolis}
	Suppose $f\colon M\to N$ is a morphism in $\modp{\cAR}$. 
	Then $f$ is a stable equivalence if and only if its base change 
	$f/(\rho,\tau)$ induces isomorphisms on Margolis homologies:
	\[
		H(M/(\rho,\tau);\alpha)=H(N/(\rho,\tau);\alpha)
	\]
	for all $\alpha\in\Set*{Q_0,Q_1,\Sq^2}$.
\end{cor}

\begin{proof}
	Note that we may always enlarge $M$ by taking direct sum with a free $\cAR$-module to obtain an epimorphism to $N$ which restricts to $f$ on $M$. 
	Hence, we may assume that $f$ is surjective. 
	Since $M$ is projective over $\MtwoR$, we see that $\Ker[f]$ is an object in $\modp{\cAR}$. 
	Hence, the base change to $\F_2$ gives us the following short exact sequence
	\[\begin{tikzcd}
	0 & {\Ker[f]/(\rho,\tau)} & {F/(\rho,\tau)} & {M/(\rho,\tau)} & 0
	\arrow[from=1-1, to=1-2]
	\arrow[from=1-2, to=1-3]
	\arrow[from=1-3, to=1-4,"{f/(\rho,\tau)}"]
	\arrow[from=1-4, to=1-5]
	\end{tikzcd}\]
	The long exact sequences of Margolis homologies thus show that 
	$f/(\rho,\tau)$ induces isomorphism on Margolis homologies 
	if and only if $\Ker[f]/(\rho,\tau)$ has vanishing Margolis homologies.
	By \cref{thm:Margolis}, this is the case if and only if $\Ker[f]$ is a free $\cAR$-module. 
\end{proof}

The Margolis homology provides us a strategy to detect stable  invertiblity.
\begin{thm}\label{Margolisinvert}
	Let $M$ be an object in $\modp{\cAR}$. 
	Then $M$ is stably invertible if and only if 
	$M/(\rho,\tau)$ has one-dimentional Margolis homologies with respect to $Q_0$, $Q_1$, and $\Sq^2$.
\end{thm}

\begin{proof}
	The proof is essentially the same as \cite{gheorghe2018picard}*{Propersition 4.11}. 
	
	Suppose $M$ is stably invertible. 
	Namely, there exists $N$ such that $M\tensor{\MtwoR}N$ is stably equivalent to $\MtwoR$. 
	Then we have isomorphisms on Margolis homologies 
	\begin{align*}		
		\MoveEqLeft
		H(M/(\rho,\tau);\alpha)\otimes H(N/(\rho,\tau);\alpha)\\
		&\iso 
		H(M/(\rho,\tau)\tensor{}N/(\rho,\tau);\alpha)\\
		&\iso 
		H((M\tensor{\MtwoR}N)/(\rho,\tau);\alpha)\\
		&\iso
		H(\F_2;\alpha),
	\end{align*}
	where the first one follows from the Ku\"nneth formula, 
	the second follows from \cref{basechange}, and 
	the last follows from \cref{cor:Margolis} when 
	$\alpha\in\Set*{Q_0,Q_1,\Sq^2}$. 
	Thus, $H(M/(\rho,\tau);\alpha)$ is one-dimentional since $H(\F_2;\alpha)$ is one-dimensional.

	Now we assume $H(M/(\rho,\tau);\alpha)$ is one-dimensional for $\alpha\in\Set*{Q_0,Q_1,\Sq^2}$. 
	Since 
	\[
		H(DM/(\rho,\tau);\alpha)
		\iso
		H(D(M/(\rho,\tau));\alpha)
		\iso 
		\Hom_{\F_2}(H(M/(\rho,\tau);\alpha), \F_2),
	\]
	we see that $DM$ also has one-dimentional Margolis homologies. 
	Then
	\[
		H(M/(\rho,\tau);\alpha)\otimes H(DM/(\rho,\tau);\alpha)
		\longrightarrow
		H(\F_2;\alpha)
	\]
	has to be an isomorphism by comparing the dimensions. 
	Then, by \cref{cor:Margolis}, $M\tensor{\MtwoR}DM\to \MtwoR$ is a stable equivalence. 
	Hence, $M$ is stably invertible.
\end{proof}


\section{The Picard group of \texorpdfstring{$\cAR$}{A(1)R}}\label{sec:pic}
This section devotes to compute the Picard group $\Pic(\cAR)$.

First, let's use \cref{Margolisinvert} to show that the $\cAR$-module $J$ is stably invertible.

\begin{prop}\label{prop:joker}
The $\cAR$-module $J$ is stably invertible, and its class $[J]$ has infinite order in $\Pic(\cAR)$. 
\end{prop}

\begin{proof}
	Recall the joker $J$ in \cref{fig:joker}. 
	Its base change to $\F_2$ has the following form:
	\begin{figure}[h]
		\[
		\begin{tikzpicture}
		\begin{scope}[ thick, every node/.style={sloped,allow upside down}, scale=0.5]
		\draw (3,0)  node[inner sep=0] (v30) {} -- (3,1) node[inner sep=0] (v31) {};
		\draw  (2.5,2) node[inner sep=0] (v22) {};
		\draw (3,3)  node[inner sep=0] (v33) {} -- (3,4) node[inner sep=0] (v34) {};
		\draw [color=blue] (v31) to [bend right=70] (v33);
		\draw [color=blue] (v30) to [bend left = 50] (v22);
		\filldraw (v30) circle (2.5pt);
		\filldraw (v31) circle (2.5pt);
		\filldraw (v22) circle (2.5pt);
		\filldraw (v33) circle (2.5pt);
		\filldraw (v34) circle (2.5pt);
		\draw (0,0) node[left]{$ $} (0,1) node[left]{$ $} (2,2) node[above]{$\bar{a}$} (1,2) node[left]{$ $};
		\draw (3,4) node[right]{$\bar{b}$} (3,3) node[right]{$ $} (2,3) node[left]{$ $} (2,2) node[right]{$ $};
		\end{scope}\end{tikzpicture}
		\]
		\caption{The $\cAR/(\rho, \tau)$-module $J/(\rho,\tau)$.} 
	\end{figure}
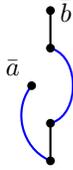

	We consider the Margolis homologies of $J/(\rho,\tau)$ with respect to $Q_0$, $Q_1$ and $\Sq^2$. 
	The $Q_0$ and $Q_1$-Margolis homologies are generated by $\bar{a}$. 
	The $\Sq^2$-Margolis homology is generated by $\bar{b}$. 
	Hence, by \cref{Margolisinvert}, $J$ is invertible.

	In order to see that $[J]$ has infinite order, notice that $a$ and $b$ are in different degrees. 
	Therefore, the degrees of the generators of $H(J^{\otimes n};Q_0)$ and $H(J^{\otimes n}; \Sq^2)$ are different for any tensor power of $J$.  
	Meanwhile, the generators of $H(\MtwoR;Q_0)$ and $H(\MtwoR; \Sq^2)$ are in the same degree. 
	Therefore, the tensor power $J^{\otimes n}$ is never stably equivalent to $\MtwoR$ since there's no morphism can induce isomorphisms on their Margolis homologies.
\end{proof}

\begin{remark}
Classically, the joker $J$ induces an order 2 element $[J]$ in $\Pic(\cA(1))$. Namely, $J^2$ is stably equivalent to $\F_2$.
\end{remark}

We now take a look at the Picard group of $\cAR/(\rho,\tau)\iso \cAC/\tau$.

\begin{lem}[\cite{gheorghe2018picard}*{Lemma 4.3}]
	As ungraded Hopf algebras, $\cAR/(\rho,\tau)$ is isomorphic to the group algebra $\F_2[D_8]$, where $D_8$ is the dihedral group of order $8$. 
	More explicitely, 
	\[
		\cAR/(\rho,\tau)\iso 
		\frac{\F_2[\Sq^1,\Sq^2]}{\Sq^1\Sq^1,\Sq^2\Sq^2,\Sq^1\Sq^2\Sq^1\Sq^2+\Sq^2\Sq^1\Sq^2\Sq^1}.
	\] 
\end{lem}

\begin{thm}\label{PicA(1)R}
The Picard group of $\cAR$ is a free abelian group of rank $4$. 
It is generated by $[\Sigma_s]$, $[\Sigma_w]$, $[I]$, and $[J]$.
\end{thm}

\begin{proof}
	Consider the composition 
	\[
		\Z^4\longmono
		\Pic(\cAR)\longrightarrow
		\Pic(\cAR/(\rho,\tau))\iso\Pic(\F_2[D_8]),
	\] 
	where the first inclusion comes from \cref{lem:PicZ3,prop:joker}, 
	and the second homomorphism comes from from \cref{injPic} and is injective. 
	Note that the ungraded Picard group of $\F_2[D_8]$ is isomorphic to $\Z^2$ with generators given by $[\Omega\F_2]$ and a module $L$ \cite{carlson2000torsion}*{Theorem 5.4}. 
	Adding the motivic bigrading, we see that the graded Picard group of $\F_2[D_8]$ is isomorphic to $\Z^4$. 
	
	By direct computation, the composition sends the joker $J$ to $[\Omega L]$. 
	Hence, it is an isomorphism.
	Then $\Pic(\cAR)\to\Pic(\cAR/(\rho,\tau))Q$, and thus the inclusion $\Z^4\hookrightarrow\Pic(\cAR)$ are also isomorphisms.
\end{proof}

\section{Generalization and other application}

\subsection{Picard group of the $\R$-motivic finite Steenrod subalgebra $\cAR[n]$}\label{sec:final}
The Picard group of the finite Steenrod subalgebra $\cA(2)$ is computed in \cite{bhattacharya2017stable}, and has been generalized to all integers in \cite{pan2022stable}. 
The authors expect to obtain the similar results for $\C$-motivic and $\R$-motivic Steenrod subalgebras $\cAC[n]$ and $\cAR[n]$.

We have injections
\[
	\Z^3\longmono
	\Pic(\cAR[n])\longrightarrow
	\Pic(\cAC[n])\longrightarrow
	\Pic(\cA(n)^\R/(\rho,\tau)).
\]

However, the computation of the $\R$-motivic Picard group relies on the Picard group of $\cA(n)^\R/(\rho,\tau)$, which is not exactly the classical Steenrod subalgebra $\cA(n)$. 
Hence, there are not existing results directly apply.


\end{document}